\documentclass[a4paper,twoside,12pt]{article}
\usepackage{amssymb,amsmath,amsthm,latexsym}
\usepackage{amsfonts}
\usepackage{graphicx,caption,subcaption}
\usepackage[pdftex,bookmarks,colorlinks=false]{hyperref}
\usepackage[hmargin=1.25in,vmargin=1.25in]{geometry}
\usepackage[mathscr]{euscript}

\usepackage[auth-sc-lg,affil-it]{authblk}
\allowdisplaybreaks

\def\ni{\noindent}

\newcommand{\bG}{\breve{G}}
\newcommand{\bP}{\breve{P}}
\newcommand{\bC}{\breve{C}}
\newcommand{\bK}{\breve{K}}
\newcommand{\bF}{\breve{F}}
\newcommand{\bW}{\breve{W}}

\newcommand{\E}{\mu}
\newcommand{\C}{\mathcal{C}}
\newcommand{\V}{\sigma^2}
\newcommand{\cE}{\E_{\chi}}
\newcommand{\cV}{\V_{\chi}}
\newcommand{\dE}{\E_{\chi^+}}
\newcommand{\dV}{\V_{\chi^+}}

\newtheorem{theorem}{Theorem}[section]
\theoremstyle{definition}
\newtheorem{definition}{Definition}
\newtheorem{problem}{Problem}

\title{\textbf{\sc On Certain Colouring Parameters of Mycielski Graphs of Some Graphs}}

\author{K. P. Chithra}
\affil{\small Naduvath Mana, Nandikkara\\ Thrissur, Kerala, India.\\ {\tt chithrasudev@gmail.com}}

\author{N. K. Sudev\footnote{Corresponding Author}, S. Satheesh}
\affil{\small Centre for Studies in Discrete Mathematics \\ Vidya Academy of Science \& Technology \\ Thrissur, Kerala, India.\\ {\tt sudevnk@gmail.com} (N. K. Sudev),\\ {\tt ssatheesh1963@yahoo.co.in} (S. Satheesh)}

\author{K. A. Germina}
\affil{\small Department of Mathematics \\ Central University of Kerala \\  Periya, Kasaragod, India.\\ {\tt srgerminaka@gmail.com}}

\author{Johan Kok}
\affil{\small Tshwane Metropolitan Police Departmet \\ City of Tshwane, South Africa.\\  {\tt kokkiek2@tshwane.gov.za}}

\date{}

\begin{document}
\maketitle

\begin{abstract}
Colouring the vertices of a graph $G$ according to certain conditions can be considered as a random experiment and a discrete random variable $X$ can be defined as the number of vertices having a particular colour in the proper colouring of $G$. The concepts of mean and variance, two important statistical measures, have also been introduced to the theory of graph colouring and determined the values of these parameters for a number of standard graphs. In this paper, we discuss the colouring parameters of the Mycielskian of certain standard graphs.
\end{abstract}

\vspace{0.2cm}

\ni \textbf{Key words}: Graph Colouring; $\chi$-chromatic mean; $\chi$-chromatic variance; $\chi^+$-chromatic mean; $\chi^+$-chromatic variance.

\vspace{0.04in}
\ni \textbf{Mathematics Subject Classification:} 05C15, 05C75, 62A01.

\section{Introduction}

For all  terms and definitions, not defined specifically in this paper, we refer to \cite{BLS,BM1,CZ,FH1,DBW} and for graph colouring, see \cite{CZ1,JT1,MK1}. For the concepts in Statistics, please see \cite{MK0,VS1,SMR1}. Unless mentioned otherwise, all graphs considered here are simple, finite, non-trivial and connected.

If $r$ is a positive integer, the $r$-th power of $G$, denoted by $G^r$, is a graph with the same vertex set such that two vertices are adjacent in $G^r$  if only if the distance between them is at most $r$.

The following theorem on graph powers plays an important role in our present study.

\begin{theorem}\label{Thm-GP}
\textup{\cite{FH1}} If $d$ is the diameter of a graph $G$, then $G^d$ is a complete graph.
\end{theorem}

An \textit{independent set} of a graph $G$ is a subset $I$ of the vertex set $V(G)$, such that no two elements (vertices) in $I$ are adjacent.

\subsection{Mycielskian of a Graph}

Consider a graph $G$ with $V(G)=\{v_1,v_2,v_3,\ldots, v_n\}$. Apply the following steps to the graph $G$.

\begin{enumerate}\itemsep0mm
\item[(i)] Take the set of new vertices $U=\{u_1,u_2,u_3,\ldots, u_n\}$ and add edges from each vertex $u_i$ of $U$ to the vertices $v_j$ if the corresponding vertex $v_i$ is adjacent to $v_j$ in $G$,
\item[(ii)] Take another new vertex $u$ and add edges to all elements in $U$. 
\end{enumerate}

The new graph thus obtained is called the \textit{Mycielski graph} or {\em or Mycielskian} of $G$ and is denoted by $\mu(G)$ (see \cite{LWLG}). For the ease of the notation in context of graph powers, we denote the Mycielski graph of a graph $G$ by $\bG$. 

The following figures illustrate the Mycielski graphs or Mycielskian of a path and a cycle.

\begin{figure*}[h!]
\centering
\begin{subfigure}[b]{0.5\textwidth}
\centering
\includegraphics[height=2in]{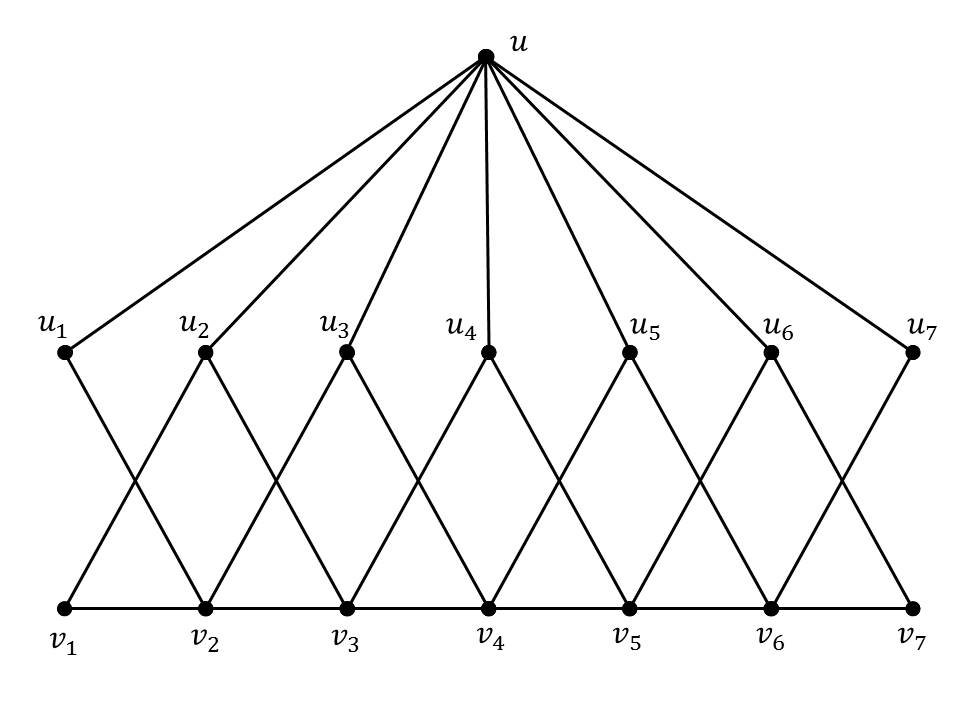}
\caption{Mycielski graph of $P_7$.}\label{Fig-1a}
\end{subfigure}%
\qquad 
\begin{subfigure}[b]{0.5\textwidth}
\centering
\includegraphics[height=2in]{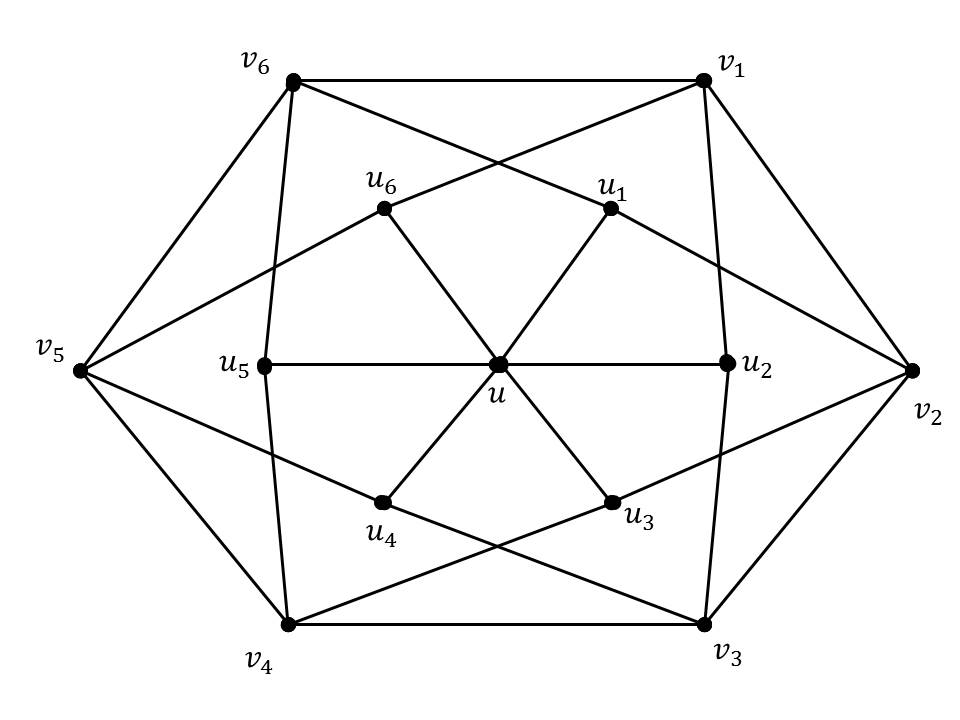}
\caption{Mycielski graph of $C_6$.}\label{Fig-1b}
\end{subfigure}
\caption{}\label{Fig-1}
\end{figure*}

A structural analysis of different powers Mycielski graphs and the adjacency pattern of vertices in these graphs have been done in \cite{SCG1}. The independent sets in different powers of the Mycielskian of paths and cycles have also been determined in \cite{SCG1}.  

\subsection{Colouring Parameters of Graphs}

The vertex colouring or simply a colouring of a graph is an assignment of colours or labels to the vertices of a graph subject to certain conditions. In a proper colouring of a graph, its vertices are coloured in such a way that no two adjacent vertices in that graph have the same colour. The first and the most important parameter in the theory of graph colouring is the \textit{chromatic number} of graphs which is defined as the minimum number of colours required in a proper colouring of the given graph.

The notion of chromatic sums of graphs related to various graph colourings have been introduced and studied extensively. Some of these studies can be found in \cite{KSC1,ES1,EK1}. Further, the notion of optimal colouring sums of a graph have been introduced and studied in \cite{KSC1,SCK1}. 

We can identify the colouring of the vertices of a given graph $G$ with a random experiment (see \cite{SSCK1}). Let $\C = \{c_1,c_2, c_3, \ldots,c_k\}$ be a proper $k$-colouring of $G$ and let $X$ be the random variable which denotes the number of vertices in $G$ having a particular colour. Since the sum of all weights of colours of $G$ is the order of $G$, the real valued function $f(i)$ defined by 
$$f(i)=P(X=i)= 
\begin{cases}
\frac{\theta(c_i)}{|V(G)|}; &  i=1,2,3,\ldots,k\\
0; & \text{elsewhere}
\end{cases}$$ 
is the probability mass function (\textit{p.m.f})of the random variable $X$. If the context is clear, we can also say that $f(i)$ is the \textit{p.m.f} of the graph $G$ with respect to the given colouring $\C$.

Hence, analogous to the definitions of the mean and variance of random variables, the mean and variance of a graph $G$, with respect to a general colouring of $G$ can be defined as follows.

\begin{definition}\label{Defn-1.1}\textup{
\cite{SSCK1} Let $\C = \{c_1,c_2, c_3, \ldots,c_k\}$ be a certain type of proper $k$-colouring of a given graph $G$ and $f(i)$ denotes the \textit{p.m.f} of a particular colour $c_i$ is assigned to vertices of $G$. Then, 
\begin{enumerate}\itemsep0mm
\item[(i)] the \textit{colouring mean} of a colouring $\C$ of a given graph $G$, denoted by $\E_{\C}(G)$ or simply by $\E(G)$, is defined to be $\E_{\C}(G)=E(i)=\sum\limits_{i=1}^{k}i\,f(i)$ and 
\item[(ii)] the \textit{colouring variance} of $\C$ is defined as $\V_{\C}(G)=V(i)=\sum\limits_{i=1}^{k}(i-\mu)^2\,f(i)$.
\end{enumerate}}
\end{definition}

If the context is clear, we can refer to these parameters as the colouring mean and colouring variance of the graph $G$ with respect to the coloring $\C$. Following the above mentioned studies, the colouring parameters of different types of vertex colouring of graphs have been done in \cite{SSCK2,SSCK3,SSCK4}. Motivated from these studies, in this paper, we investigate the colouring parameters of the Mycielskian of certain standard graph classes.

\section{$\chi$-Chromatic Parameters of the Mycielskian of Certain Graphs}

A colouring mean of a graph $G$, with respect to a proper colouring $\C$ is said to be a \textit{$\chi$-chromatic mean} of $G$, if $\C$ is the minimum proper colouring of $G$ and the colouring sum $\omega_G$ is also minimum (see \cite{SSCK1}). The $\chi$-chromatic mean of a graph $G$ is denoted by $\cE(G)$. The \textit{$\chi$-chromatic variance} of $G$, denoted by $\cV(G)$, is a colouring variance of $G$ with respect to a minimal proper colouring $\C$ of $G$ which yields the minimum colouring sum.

Throughout this study, we represent the set $V=\{v_1,v_2,v_3,\ldots,v_n\}$ as the vertex set of the graph $G$ (in $\bG$), $U=\{u_1,u_2,u_3,\ldots,u_n\}$ be the set of the newly introduced (independent) vertices in $\bG$ and $u$ be the vertex which is adjacent to the vertices in $U$ in $\bG$. In the first result, we determine the $\chi$-chromatic parameters of the Mycielskians of paths.

\begin{theorem}\label{Thm-2.1}
The $\chi$-chromatic mean of the Mycielskian of a path $P_n$ is 
\begin{equation*}
\cE(\bP_n)=
\begin{cases}
\frac{7n+3}{4n+2}; & \text{if $n$ is odd},\\
\frac{7n+4}{4n+2}; & \text{if $n$ is even},
\end{cases}
\end{equation*}
and the $\chi$-chromatic variance of the Mycielskian of a path $P_n$ is
\begin{equation*}
\cV(\bP_n)=
\begin{cases}
\frac{11n^2-3}{(4n+2)^2}; & \text{if $n$ is odd},\\
\frac{11n^2+6n}{(4n+2)^2}; & \text{if $n$ is even}.
\end{cases}
\end{equation*}
\end{theorem}
\begin{proof}
Consider the Mycielskian $\bP_n$ of a path $P_n$. It has $2n+1$ vertices. Clearly, the set $U$ is the largest independence set in $\bP_n$. Hence, we can assign colour $c_1$ to all vertices in $U$. Now, every vertex in $V$ is adjacent to at least one vertex in $U$, no vertices in $V$ can have the colour $c_1$. Since $P_n$ is $2$-colourable, $\lceil\frac{n}{2}\rceil$ vertices in $V$ can be assigned the colour $c_2$, while the remaining $\lfloor\frac{n}{2}\rfloor$ vertices can have the colour $c_3$. Since $u$ is adjacent to all vertices of $U$ and $u$ is not adjacent to any vertex in $V$, $u$ can have the colour $c_2$. If $\theta(c_i)$ denotes the cardinality of the colour class of $c_i$, then we have $\theta(c_1)=n$, $\theta(c_2)=1+\lceil\frac{n}{2}\rceil$ and $\theta(c_3)=\lfloor\frac{n}{2}\rfloor$. Then, we have the following cases.

\textit{Case-1:} Let $n$ be odd. Then, we have $\theta(c_1)=n$, $\theta(c_2)=1+\frac{n+1}{2}=\frac{n+3}{2}$ and $\theta(c_3)=\frac{n-1}{2}$. Now, define the function 
\begin{equation*}
f(i)=
\begin{cases}
\frac{n}{2n+1}; & \text{if}\ i=1,\\
\frac{n+3}{4n+2}; & \text{if}\ i=2,\\
\frac{n-1}{4n+2}; & \text{if}\ i=3.
\end{cases}
\end{equation*}
Then, $\sum\limits_{i=1}^{3}f(i)= \frac{n}{2n+1}+\frac{n+3}{4n+2}+\frac{n-1}{4n+2}=1$ and hence $f(i)$ is the \textit{p.m.f} of $\bP_n$ under the colouring concerned. Therefore, $\chi$-chromatic mean of $\bP_n$ is given by $\cE(\bP_n)=1\cdot(\frac{n}{2n+1})+2\cdot(\frac{n+3}{4n+2})+3\cdot(\frac{n-1}{4n+2})=\frac{7n+3}{4n+2}$. Also, $\chi$-chromatic variance of $\bP_n$ is given by
\begin{eqnarray*}
\cV(\bP_n) & = & \sum(i-\cE(\bP_n))^2\,f(i)\\
& = & \left(1-\frac{7n+3}{4n+2}\right)^2\cdot \frac{n}{2n+1}+ \left(2-\frac{7n+3}{4n+2}\right)^2\cdot \frac{n+3}{4n+2}+ \\ & & \left(3-\frac{7n+3}{4n+2}\right)^2\cdot \frac{n-1}{4n+2}\\
& = & \frac{44n^3+22n^2-12n-6}{(4n+2)^3}\\
& = & \frac{11n^2-3}{(4n+2)^2}.
\end{eqnarray*}

\textit{Case-2:} Let $n$ be even. Then, we have $\theta(c_1)=n$, $\theta(c_2)=1+\frac{n}{2}=\frac{n+2}{2}$ and $\theta(c_3)=\frac{n}{2}$. Now, define the function 
\begin{equation*}
f(i)=
\begin{cases}
\frac{n}{2n+1}; & \text{if}\ i=1,\\
\frac{n+2}{4n+2}; & \text{if}\ i=2,\\
\frac{n}{4n+2}; & \text{if}\ i=3.
\end{cases}
\end{equation*}
Then, $\sum\limits_{i=1}^{3}f(i)= \frac{n}{2n+1}+\frac{n+2}{4n+2}+\frac{n}{4n+2}=1$ and hence $f(i)$ is the \textit{p.m.f} of $\bP_n$ under the colouring concerned. Therefore, $\chi$-chromatic mean of $\bP_n$ is given by $\cE(\bP_n)=1\cdot(\frac{n}{2n+1})+2\cdot(\frac{n+2}{4n+2})+3\cdot(\frac{n}{4n+2})=\frac{7n+4}{4n+2}$. Also, $\chi$-chromatic variance of $\bP_n$ is given by
\begin{eqnarray*}
\cV(\bP_n) & = & \sum(i-\cE(\bP_n))^2\,f(i)\\
& = & \left(1-\frac{7n+4}{4n+2}\right)^2\cdot \frac{n}{2n+1}+ \left(2-\frac{7n+4}{4n+2}\right)^2\cdot \frac{n+2}{4n+2}+ \\ & & \left(3-\frac{7n+4}{4n+2}\right)^2\cdot \frac{n}{4n+2}\\
& = & \frac{44n^3+46n^2+12n}{(4n+2)^3}\\
& = & \frac{11n^2+6n}{(4n+2)^2}.
\end{eqnarray*}
\end{proof}

Next, we discuss the $\chi$-chromatic parameters of the Mycielskian of cycles $C_n$ in the following theorem. 

\begin{theorem}\label{Thm-2.3}
The $\chi$-chromatic mean of the Mycielskian of a cycle $C_n$ is 
\begin{equation*}
\cE(\bC_n)=
\begin{cases}
\frac{7n+4}{4n+2}; & \text{if $n$ is even},\\
\frac{7n+7}{4n+2}; & \text{if $n$ is odd},
\end{cases}
\end{equation*}
and the $\chi$-chromatic variance of the Mycielskian of $C_n$ is
\begin{equation*}
\cV(\bC_n)=
\begin{cases}
\frac{11n^2+6n}{(4n+2)^2}; & \text{if $n$ is even},\\
\frac{44n^3+182n^2+100n+10}{(4n+2)^3}; & \text{if $n$ is odd}.
\end{cases}
\end{equation*}
\end{theorem}
\begin{proof}
As explained in Theorem \ref{Thm-2.1}, the set $U$ is the largest independence set in $\bC_n$ also and we assign colour $c_1$ to all vertices in $U$. As stated earlier, every vertex of $V$ is adjacent to at least one vertex in $U$ and hence no vertices in $V$ can have the colour $c_1$. To proceed further, we have to consider the following cases.

\textit{Case-1:} Let $n$ be even. Then, $C_n$ can be coloured using two colours, say $c_1$ and $c_2$, the corresponding colour classes contain $\frac{n}{2}$ vertices each. Since $u$ is adjacent to all vertices in $U$ and $u$ is not adjacent to any vertex in $V$, we can assign the colour $c_2$ to  $u$. If $\theta(c_i)$ denotes the cardinality of the colour class of $c_i$, then we have $\theta(c_1)=n$, $\theta(c_2)=1+\frac{n}{2}=\frac{n+2}{2}$ and $\theta(c_3)=\frac{n}{2}$. Therefore, the corresponding \textit{p.m.f} is given by
\begin{equation*}
f(i)=
\begin{cases}
\frac{n}{2n+1}; & \text{if}\ i=1,\\
\frac{n+2}{4n+2}; & \text{if}\ i=2,\\
\frac{n}{4n+2}; & \text{if}\ i=3,
\end{cases}
\end{equation*}
which is similar to Part (ii) of Theorem \ref{Thm-2.1}. Therefore, as explained there, we have $\cE(\bC_n)=\frac{7n+4}{4n+2}$ and $\cV(\bC_n)=\frac{11n^2+6n}{(4n+2)^2}$. 

\textit{Case-2} Let $n$ be odd. Then, $C_n$ is $3$-colourable and we can colour the vertices in $V$ using three colours, say $c_2, c_3, c_4$ such that $\lfloor\frac{n}{2}\rfloor$ vertices have colours $c_2$ and $c_3$, while one vertex has colour $c_4$. The vertex $u$ can be coloured using the colour $c_2$. Hence, the corresponding \textit{p.m.f} is 
\begin{equation*}
f(i)=
\begin{cases}
\frac{n}{2n+1}; & \text{if}\ i=1,\\
\frac{n+1}{4n+2}; & \text{if}\ i=2,\\
\frac{n-1}{4n+2}; & \text{if}\ i=3,\\
\frac{1}{2n+1}; & \text{if}\ i=4,
\end{cases}
\end{equation*}
as $\sum\limits_{i=1}^4 f(i)=\frac{n}{2n+1}+\frac{n+1}{4n+2}+\frac{n-1}{4n+2}+\frac{1}{2n+1}=1$.

Therefore, the $\chi$-chromatic mean is given by $\cE(\bC_n)=1\cdot (\frac{n}{2n+1})+2\cdot (\frac{n+1}{4n+2})+3\cdot (\frac{n-1}{4n+2})+4\cdot(\frac{1}{2n+1})=\frac{7n+7}{4n+2}$ and the $\chi$-chromatic variance is given by 
\begin{eqnarray*}
\cV(\bC_n) & = & \sum(i-\cE(\bP_n))^2\,f(i)\\
& = & \left(1-\frac{7n+7}{4n+2}\right)^2\cdot \frac{n}{2n+1}+ \left(2-\frac{7n+7}{4n+2}\right)^2\cdot \frac{n+1}{4n+2}+ \\ & & \left(3-\frac{7n+7}{4n+2}\right)^2\cdot \frac{n-1}{4n+2}+ \left(4-\frac{7n+7}{4n+2}\right)^2\cdot \frac{1}{2n+1}\\
& = & \frac{44n^3+182n^2+100n+10}{(4n+2)^3}.
\end{eqnarray*}
\end{proof}

The $\chi$-chromatic parameters of complete bipartite graphs are determined in the following result. 

\begin{theorem}\label{Thm-2.5}
The $\chi$-chromatic mean of the Mycielskian of a complete bipartite graph $K_{a,b}, a>b$ is
$$\cE(\bK_{a,b})= 1+\frac{2(b+1)}{2n+1}$$   
and the $\chi$-chromatic variance of $\bK_{a,b}$ is $$\cV(\bK_{a,b})=\frac{16a^2+4b^2+8a^2b+8b^2a+24ab+8a+2b}{(2n+1)^3}.$$
\end{theorem}
\begin{proof}
For two positive integers $a$ and $b$ such that $a+b=n$, we know that the Mycielski graph $\bK_{a,b}$ is $3$-colourable. Let $(U_1,U_2)$ be the bipartition of the set $U$ corresponding to the bipartition $(V_1,V_2)$ of $K_{a,b}$. Here, we observe that $U_1\cup V_1$ and $U_2\cup V_2$ are independent sets in $\bK_{a,b}$. We also note that the colouring in which the vertices in $U_1\cup V_1$ get the colour $c_1$, the vertices in $U_2\cup V_2$ get the colour $c_2$ and the vertex $u$ gets the colour $c_3$ yields the minimum chromatic sum. In this context, define the function 
$$f(i)=
\begin{cases}
\frac{2a}{2n+1}; & \text{if}\ i=1,\\
\frac{2b}{2n+1}; & \text{if}\ i=2,\\
\frac{1}{2n+1}; & \text{if}\ i=3.
\end{cases}$$
Here, $\sum\limits_{i=1}^3 f(i)=\frac{2a}{2n+1}+\frac{2b}{2n+1}+\frac{1}{2n+1}=\frac{2(a+b)+1}{2n+1}=1$. Therefore, $f(i)$ is the \textit{p.m.f} of the graph $\bK_{a,b}$ with respect to the colouring concerned.

Then, the $\chi$-chromatic mean of $\bK_{a,b}$ is given by $\cE(\bK_{a,b})=1\cdot\left(\frac{2a}{2n+1}\right)+2\cdot\left(\frac{2b}{2n+1}\right)+3\cdot\left(\frac{1}{2n+1}\right)=\frac{2a+4b+3}{2n+1}= 1+\frac{2(b+1)}{2a+2b+1}$. 

The $\chi$-chromatic variance of $\bK_{a,b}$ is given by 
\begin{eqnarray*}
\cV(\bK_{a,b}) & = & \left(1-\frac{2a+4b+3}{2n+1}\right)^2\cdot\left(\frac{2a}{2n+1}\right)+\left(2-\frac{2a+4b+3}{2n+1}\right)^2\cdot\left(\frac{2b}{2n+1}\right)+\\ & & \left(3-\frac{2a+4b+3}{2n+1}\right)^2\cdot\left(\frac{1}{2n+1}\right)\\
& = & \frac{16a^2+4b^2+8a^2b+8b^2a+24ab+8a+2b}{(2n+1)^3}.
\end{eqnarray*} 
\end{proof}

Next, we determine the $\chi$-chromatic parameters of the Mycielskian of complete graphs in the following theorem.

\begin{theorem}\label{Thm-2.6}
The $\chi$-chromatic mean of the Mycielskian of a complete graph $K_n$ is
$$\cV(\bK_n)= \frac{n^2+5n+4}{2n+1}$$ and the $\chi$-chromatic variance of $\bK_n$ is $$\cV(\bK_n)=\frac{7n^4+30n^3+61n^2+94n+32}{12(2n+1)^2}.$$
\end{theorem}
\begin{proof}
As mentioned in earlier theorems, the vertex set $U$ is the largest independence set in $\bK_n$ and we can assign the colour $c_1$ to all vertices in $U$. Subsequently one vertex in $V$ and the vertex $u$ can be coloured using the colour $c_2$ and all other vertices in $V$ must have distinct colours, say $c_3,c_4,\ldots,c_{n+1}$. Now, define a function
$$f(i)=
\begin{cases}
\frac{n}{2n+1}; & \text{if}\ i=1,\\
\frac{2}{2n+1}; & \text{if}\ i=2,\\
\frac{1}{2n+1}; & \text{if}\ i=3,4,5,\ldots, n+1.
\end{cases}$$
Here, $\sum\limits_{1=1}^{n+1}f(i)=\frac{n}{2n+1}+\frac{2}{2n+1}+(n-1)(\frac{1}{2n+1})=1$ and hence $f(i)$ is the $p. m. f$ of the graph $\bK_n$. 

Then, the $\chi$-chromatic mean of the graph $\bK_n$ is given by
\begin{eqnarray*}
\cE(\bK_n) & = & 1\cdot(\frac{n}{2n+1})+2\cdot(\frac{2}{2n+1})+(3+4+\ldots+(n+1))(\frac{1}{2n+1})\\
& = & \frac{n}{2n+1} + \frac{4}{2n+1} + \frac{(n-1)(n+4)}{4n+2}\\
& = & \frac{n^2+5n+4}{4n+2},
\end{eqnarray*}
and the $\chi$-chromatic variance of $\bK_n$ is given by,
\begin{eqnarray*}
\cV(\bK_n) & = & E(i-\cE(\bK_n))\\
& = & E(i^2)-(E(i))^2\\
& = & \left[1^2\cdot(\frac{n}{2n+1})+2^2\cdot(\frac{2}{2n+1})+(3^2+4^2+\ldots+(n+1)^2)(\frac{1}{2n+1})\right]-\\ & &  \left(\frac{n^2+5n+4}{4n+2}\right)^2\\
& = & \frac{2n^3+9n^2+19n+24}{6(2n+1)}-\left(\frac{n^2+5n+4}{4n+2}\right)^2 \\
& = & \frac{7n^4+30n^3+61n^2+94n+32}{12(2n+1)^2}.
\end{eqnarray*}
\end{proof}

Another interesting graph class we consider next is the \textit{wheel graph} $W_{n+1}$, which obtained by drawing edges from all vertices of a cycle $C_n$ to an external vertex. In the following theorem, we discuss the $\chi$-chromatic parameters of the Mycielskian of a wheel graph $W_{n+1}$.

\begin{theorem}\label{Thm-2.7}
The $\chi$-chromatic mean of the Mycielskian of a wheel graph $W_{n+1}$ is given by
$$\cE(\bW_{n+1})=
\begin{cases}
\frac{7n+14}{4n+6}; & \text{if $n$ is even},\\
\frac{7n+19}{4n+6}; & \text{if $n$ is odd},
\end{cases}$$
and the $\chi$-chromatic variance of $\bW_{n+1}$
$$\cV(\bW_{n+1})=
\begin{cases}
\frac{11n^2+62n+56}{(4n+6)^2}; & \text{if $n$ is even},\\
\frac{44n^3+626n^2+1292n+678}{(4n+6)^3}; & \text{if $n$ is odd}.
\end{cases}$$
\end{theorem}
\begin{proof}
Note that vertex sets $U$ and $V$ contains $n+1$ vertices each and hence $\bW_{n+1}$ has $2n+3$ vertices. As usual, the vertices in $U$ are independent to each other and all of them can have the colour $c_1$. Now, we have to consider the following cases.

\textit{Case-1:} Let $n$ be even. Then, $\frac{n}{2}$ vertices in $V$ and the vertex $u$ can have the colour $c_2$ and $\frac{n}{2}$ vertices in $V$  can have colour $c_3$ and the central vertex of $W_{n+1}$ can have the colour $c_4$. Then, define the function 
$$f(i)=
\begin{cases}
\frac{n+1}{2n+3}; & \text{if}\ i=1,\\
\frac{n+2}{4n+6}; & \text{if}\ i=2,\\
\frac{n}{4n+6}; & \text{if}\ i=3,\\
\frac{1}{2n+3}; & \text{if}\ i=4.\\
\end{cases}$$
Clearly, $\sum\limits_{i=1}^4f(i)=1$ and $f(i)$ is the \textit{p.m.f} of the graph $\bW_{n+1}$ with respect to the colouring concerned. Therefore, we have $\cE(\bW_{n+1})=1\cdot(\frac{n+1}{2n+3})+2\cdot(\frac{n+2}{4n+6})+3\cdot(\frac{n}{4n+6})+4\cdot(\frac{1}{2n+3})=\frac{7n+14}{4n+6}$ and 
\begin{eqnarray*}
\cV(\bW_{n+1}) & = & \sum\limits_{i=1}^4(i-\cE(\bW_{n+1}))^2f(i)\\
& = & \left(1-\frac{7n+14}{4n+6}\right)^2\cdot(\frac{n+1}{2n+3})+\left(2-\frac{7n+14}{4n+6}\right)^2\cdot(\frac{n+2}{4n+6})+\\ & & \left(3-\frac{7n+14}{4n+6}\right)^2\cdot(\frac{n}{4n+6})+\left(4-\frac{7n+14}{4n+6}\right)^2\cdot(\frac{1}{2n+3})\\
& = & \left(\frac{3n+8}{4n+6}\right)^2\cdot \left(\frac{n+1}{2n+3}\right)+ \left(\frac{n-2}{4n+6}\right)^2\cdot \left(\frac{n+2}{4n+6}\right)+\\
& & \left(\frac{5n+4}{4n+6}\right)^2\cdot \left(\frac{n}{4n+6}\right)+ \left(\frac{9n+10}{4n+6}\right)^2\cdot \left(\frac{1}{2n+3}\right) \\
& = & \frac{11n^2+62n+56}{(4n+6)^2}.
\end{eqnarray*}

\textit{Case-2:} Let $n$ be odd. Then, $\frac{n-1}{2}$ vertices in $V$ and the vertex $u$ can have the colour $c_2$ and $\frac{n-1}{2}$ vertices in $V$ can have colour $c_3$ and the remaining vertex and the central vertex of $W_{n+1}$ can have the colours $c_4$ and $c_5$ each. Then, define the function 
$$f(i)=
\begin{cases}
\frac{n+1}{2n+3}; & \text{if}\ i=1,\\
\frac{n+1}{4n+6}; & \text{if}\ i=2,\\
\frac{n-1}{4n+6}; & \text{if}\ i=3,\\
\frac{1}{2n+3}; & \text{if}\ i=4,5.\\
\end{cases}$$
Clearly, $\sum\limits_{i=1}^5f(i)=1$ and $f(i)$ is the \textit{p.m.f} of the graph $\bW_{n+1}$ with respect to the colouring concerned. Therefore, we have $\cE(\bW_{n+1})=1\cdot(\frac{n+1}{2n+3})+2\cdot(\frac{n+1}{4n+6})+3\cdot(\frac{n-1}{4n+6})+(4+5)\cdot(\frac{1}{2n+3})=\frac{7n+19}{4n+6}$ and

\begin{eqnarray*}
\cV(\bW_{n+1}) & = & \sum\limits_{i=1}^4(i-\cE(\bW_{n+1}))^2f(i)\\
& = & \left(1-\frac{7n+19}{4n+6}\right)^2\cdot\left(\frac{n+1}{2n+3}\right)+\left(2-\frac{7n+19}{4n+6}\right)^2\cdot\left(\frac{n+1}{4n+6}\right)+\\ & & \left(3-\frac{7n+19}{4n+6}\right)^2\cdot\left(\frac{n-1}{4n+6}\right)+\left(4-\frac{7n+19}{4n+6}\right)^2\cdot\left(\frac{1}{2n+3}\right)+\\ & & \left(5-\frac{7n+19}{4n+6}\right)^2\cdot\left(\frac{1}{2n+3}\right)\\
& = & \left(\frac{3n+13}{4n+6}\right)^2\cdot \left(\frac{n+1}{2n+3}\right)+ \left(\frac{n-7}{4n+6}\right)^2\cdot \left(\frac{n+1}{4n+6}\right)+ \\ 
& & \left(\frac{5n-1}{4n+6}\right)^2\cdot\left(\frac{n-1}{4n+6}\right)
+\left(\frac{9n+5}{4n+6}\right)^2\cdot \left(\frac{1}{2n+3}\right)+ \\ 
& & \left(\frac{13n+11}{4n+6}\right)^2\cdot \left(\frac{1}{2n+3}\right) \\
& = & \frac{44n^3+626n^2+1292n+678}{(4n+6)^3}.
\end{eqnarray*} 
\end{proof}

A \textit{friendship graph}, denoted by $F_{n+1}$ is the graph obtained by drawing edges from all vertices of a path $P_n$ to external vertex. In the following theorem, we discuss the $\chi$-chromatic parameters of the Mycielskian of friendship graphs.

\begin{theorem}\label{Thm-2.8}
The $\chi$-chromatic mean of the Mycielskian of a friendship graph $F_{n+1}$ is given by
$$\cE(\bF_{n+1})=
\begin{cases}
\frac{7n+14}{4n+6}; & \text{if $n$ is even},\\
\frac{7n+13}{4n+6}; & \text{if $n$ is odd},
\end{cases}$$
and the $\chi$-chromatic variance of $\bF_{n+1}$ is
$$\cV(\bF_{n+1})=
\begin{cases}
\frac{11n^2+62n+56}{(4n+6)^2}; & \text{if $n$ is even},\\
\frac{11n^2+56n+53}{(4n+6)^2}; & \text{if $n$ is odd}.
\end{cases}$$
\end{theorem}
\begin{proof}
The graph $\bF_{n+1}$ has $2n+3$ vertices. As usual, all $n+1$ vertices in $U$ are independent to each other and all of them can have the colour $c_1$. Now, we have to consider the following cases.

\textit{Case-1:} Let $n$ be even. Then, $\frac{n}{2}$ vertices in $V$ and the vertex $u$ can have the colour $c_2$ and $\frac{n}{2}$ vertices in $V$  can have colour $c_3$ and the central vertex of $F_{n+1}$ can have the colour $c_4$. Then, define the function 
$$f(i)=
\begin{cases}
\frac{n+1}{2n+3}; & \text{if}\ i=1,\\
\frac{n+2}{4n+6}; & \text{if}\ i=2,\\
\frac{n}{4n+6}; & \text{if}\ i=3,\\
\frac{1}{2n+3}; & \text{if}\ i=4.\\
\end{cases}$$

This is similar to Part (i) of Theorem \ref{Thm-2.7} and using the similar arguments, we have $\cE(\bF_{n+1})=\frac{7n+14}{4n+6}$ and $\cV(\bF_{n+1})=\frac{11n^2+62n+56}{(4n+6)^2}$.

\textit{Case-2:} Let $n$ be odd. Then, $\frac{n+1}{2}$ vertices in $V$ and the vertex $u$ can have the colour $c_2$ and $\frac{n-1}{2}$ vertices in $V$  can have colour $c_3$ and the central vertex of $F_{n+1}$ can have the colour $c_4$. Then, define the function 
$$f(i)=
\begin{cases}
\frac{n+1}{2n+3}; & \text{if}\ i=1,\\
\frac{n+3}{4n+6}; & \text{if}\ i=2,\\
\frac{n-1}{4n+6}; & \text{if}\ i=3,\\
\frac{1}{2n+3}; & \text{if}\ i=4,\\
\end{cases}$$
which is clearly the \textit{p.m.f} of the graph $\bF_{n+1}$ and hence we have $\cE(F_{n+1})=1\cdot(\frac{n+1}{2n+3})+2\cdot(\frac{n+3}{4n+6})+3\cdot(\frac{n-1}{4n+6})+4\cdot(\frac{1}{2n+3})=\frac{7n+13}{4n+6}$ and 
\begin{eqnarray*}
\cE(F_{n+1}) & = & E(i-\cE(\bF{n+1}))\\
& = & \left(1-\frac{7n+13}{4n+6}\right)^2\cdot \left(\frac{n+1}{2n+3}\right)+ \left(2-\frac{7n+13}{4n+6}\right)^2\cdot \left(\frac{n+3}{4n+6}\right)+ \\
& & \left(3-\frac{7n+13}{4n+6}\right)^2\cdot \left(\frac{n-1}{4n+6}\right)+ \left(4-\frac{7n+13}{4n+6}\right)^2\cdot \left(\frac{1}{2n+3}\right)\\
& = & \left(\frac{3n+7}{4n+6}\right)^2\cdot \left(\frac{n+1}{2n+3}\right)+ \left(\frac{n-1}{4n+6}\right)^2\cdot \left(\frac{n+3}{4n+6}\right)+ \\ 
& & \left(\frac{5n+5}{4n+6}\right)^2\cdot \left(\frac{n-1}{4n+6}\right)+ \left(\frac{9n+11}{4n+6}\right)^2\cdot \left(\frac{1}{2n+3}\right)\\
& = &  \frac{44n^3+290n^2+548n+318}{(4n+6)^3}\\
& = & \frac{11n^2+56n+53}{(4n+6)^2}.
\end{eqnarray*}

\end{proof}

\section{A Few Words on the $\chi^+$-Chromatic Parameters of Mycielski Graphs}

A colouring mean of a graph $G$, with respect to a proper colouring $\C$ is said to be a \textit{$\chi^+$-chromatic mean} of $G$, if $\C$ is a minimum proper colouring of $G$ such that the corresponding colouring sum $\omega_G$ is maximum. The $\chi^+$-chromatic number of a graph $G$ is denoted by $\dE(G)$ (\cite{KSC1}). The \textit{$\chi^+$-chromatic variance} of $G$, denoted by $\cV(G)$, is a colouring variance of $G$ with respect to a minimal proper colouring $\C$ of $G$ such that the corresponding colouring sum is maximum (\cite{KSC1}).

Note that we can determine the $\chi^+$-chromatic parameters of a given graph by reversing the graph colouring. Hence, we have the following inferences. 

Consider the Mycielskian of a path $P_n$. As stated earlier, $\bP_n$ is $3$-colourable. Reverse order of colouring mentioned in \ref{Thm-2.1} to get the $\chi^+$-chromatic parameters of $\bP_n$. If $n$ is even, the $p.m. f$ of $\bP_n$ is given by

\begin{equation*}
f(i)=
\begin{cases}
\frac{n}{4n+2}; & \text{if}\ i=1,\\
\frac{n+2}{4n+2}; & \text{if}\ i=2,\\
\frac{n}{2n+1}; & \text{if}\ i=3,
\end{cases}
\end{equation*}
and if $n$ is odd, the corresponding \textit{p.m.f} will be 
\begin{equation*}
f(i)=
\begin{cases}
\frac{n-1}{4n+2}; & \text{if}\ i=1,\\
\frac{n+3}{4n+2}; & \text{if}\ i=2,\\
\frac{n}{2n+1}; & \text{if}\ i=3.
\end{cases}
\end{equation*}

Hence, as explained in the above section, the $\chi^+$-chromatic parameters of the graph $\bP_n$ are as follows.

\begin{theorem}\label{Thm-3.1}
The $\chi^+$-chromatic mean of the Mycielskian of a path $P_n$ is 
\begin{equation*}
\dE(\bP_n)=
\begin{cases}
\frac{9n+5}{4n+2}; & \text{if $n$ is odd},\\
\frac{9n+4}{4n+2}; & \text{if $n$ is even},
\end{cases}
\end{equation*}
and the $\chi^+$-chromatic variance of the Mycielskian of a path $P_n$ is
\begin{equation*}
\dV(\bP_n)=
\begin{cases}
\frac{44n^3+22n^2-12n-6}{(4n+2)^3}; & \text{if $n$ is odd},\\
\frac{11n^2+6n}{(4n+2)^2}; & \text{if $n$ is even}.
\end{cases}
\end{equation*}
\end{theorem}

Now consider the Mycielskian of a cycle $C_n$ for the calculation of $\chi^+$-chromatic parameters. If $n$ is even, then by reversing the order of colouring, we get the corresponding \textit{p.m.f} as 
\begin{equation*}
f(i)=
\begin{cases}
\frac{n}{4n+2}; & \text{if}\ i=1,\\
\frac{n+2}{4n+2}; & \text{if}\ i=2,\\
\frac{n}{2n+1}; & \text{if}\ i=3,
\end{cases}
\end{equation*}
and if $n$ is odd, then the corresponding \textit{p.m.f} is 
\begin{equation*}
f(i)=
\begin{cases}
\frac{1}{2n+1}; & \text{if}\ i=1,\\
\frac{n-1}{4n+2}; & \text{if}\ i=2,\\
\frac{n+1}{4n+2}; & \text{if}\ i=3,\\
\frac{n}{2n+1}; & \text{if}\ i=4.
\end{cases}
\end{equation*}

\ni Then, the $\chi^+$-chromatic parameters of $\bC_n$ are as given in the following result.

\begin{theorem}\label{Thm-3.2}
The $\chi^+$-chromatic mean of the Mycielskian of a cycle $C_n$ is 
\begin{equation*}
\dE(\bC_n)=
\begin{cases}
\frac{9n+4}{4n+2}; & \text{if $n$ is even},\\
\frac{13n+3}{4n+2}; & \text{if $n$ is odd},
\end{cases}
\end{equation*}
and the $\chi^+$-chromatic variance of the Mycielskian of $C_n$ is
\begin{equation*}
\dV(\bC_n)=
\begin{cases}
\frac{11n^2+6n}{(4n+2)^2}; & \text{if $n$ is even},\\
\frac{44n^3+182n^2+70n+10}{(4n+2)^3}; & \text{if $n$ is odd}.
\end{cases}
\end{equation*}
\end{theorem}

For the Mycielskian of a complete bipartite graph $K_{a,b}$, reversing the coloring pattern, the corresponding \textit{p.m.f} will be obtained as
$$f(i)=
\begin{cases}
\frac{1}{2n+1}; & \text{if}\ i=1,\\
\frac{2b}{2n+1}; & \text{if}\ i=2,\\
\frac{2a}{2n+1}; & \text{if}\ i=3.
\end{cases}$$

Then, the $\chi^+$-chromatic parameters of $\bK_{a,b}$ are obtained as given in the following result.

\begin{theorem}\label{Thm-3.3}
The $\chi^+$-chromatic mean of the Mycielskian of a complete bipartite graph $K_{a,b}, a>b$ is
$$\dE(\bK_{a,b})= 2+\frac{(2a-1)}{2a+2b+1}$$   
and the $\chi^+$-chromatic variance of $\bK_{a,b}$ is $$\dV(\bK_{a,b})=\frac{16a^2+4b^2+8a^2b+8b^2a+24ab+8a+2b}{(2n+1)^3}.$$

\end{theorem}

\ni Next, we discuss the $\chi^+$-chromatic parameters of friendship graphs and wheel graphs. 

If $n$ is even, by reversing the coloring pattern of $\bF_{n+1}$, we get the corresponding \textit{p.m.f} as follows.

$$f(i)=
\begin{cases}
\frac{1}{2n+3}; & \text{if}\ i=1,\\
\frac{n}{4n+6}; & \text{if}\ i=2,\\
\frac{n+2}{4n+6}; & \text{if}\ i=3,\\
\frac{n+1}{2n+3}; & \text{if}\ i=4,

\end{cases}$$
and if $n$ is odd, then the \textit{p.m.f}  is
$$f(i)=
\begin{cases}
\frac{1}{2n+3}; & \text{if}\ i=1,\\
\frac{n-1}{4n+6}; & \text{if}\ i=2,\\
\frac{n+3}{4n+6}; & \text{if}\ i=3,\\
\frac{n+1}{2n+3}; & \text{if}\ i=4.
\end{cases}$$

Then, the $\chi^+$-chromatic parameters of the Mycielskian of the friendship graphs can be calculated as mentioned in the given result.

\begin{theorem}\label{Thm-3.4}
The $\chi^+$-chromatic mean of the Mycielskian of a friendship graph $F_{n+1}$ is given by
$$\dE(\bF_{n+1})=
\begin{cases}
\frac{13n+16}{4n+6}; & \text{if $n$ is even},\\
\frac{13n+17}{4n+6}; & \text{if $n$ is odd},
\end{cases}$$
and the $\chi^+$-chromatic variance of $\bF_{n+1}$ is given by
$$\dV(\bF_{n+1})=
\begin{cases}
\frac{11n^2+62n+56}{(4n+6)^2}; & \text{if $n$ is even},\\
\frac{11n^2+56n+53}{(4n+6)^2}; & \text{if $n$ is odd}.
\end{cases}$$
\end{theorem}

If $n$ is even, then reversing the colouring pattern of the Mycielski graphs of the wheel graphs, the new \textit{p.m.f} will be as follows.

$$f(i)=
\begin{cases}
\frac{1}{2n+3}; & \text{if}\ i=1,\\
\frac{n}{4n+6}; & \text{if}\ i=2,\\
\frac{n+2}{4n+6}; & \text{if}\ i=3,\\
\frac{n+1}{2n+3}; & \text{if}\ i=4,\\
\end{cases}$$
and if $n$ is odd, then the \textit{p.m.f} of $\bW_{n+1}$ is 

$$f(i)=
\begin{cases}
\frac{1}{2n+3}; & \text{if}\ i=1,2,\\
\frac{n-1}{4n+6}; & \text{if}\ i=3,\\
\frac{n+1}{4n+6}; & \text{if}\ i=4,\\
\frac{n+1}{2n+3}; & \text{if}\ i=5.\\
\end{cases}$$

Now, the $\chi^+$-chromatic parameters of the Mycielski graphs of the wheel graphs can be found out as given in the following theorem.  

\begin{theorem}\label{Thm-3.5}
The $\chi^+$-chromatic mean of the Mycielskian of a wheel graph $W_{n+1}$ is given by
$$\dE(\bW_{n+1})=
\begin{cases}
\frac{13n+16}{4n+6}; & \text{if $n$ is even},\\
\frac{17n+17}{4n+6}; & \text{if $n$ is odd},
\end{cases}$$
and the $\chi^+$-chromatic parameters of $\bW_{n+1}$
$$\dV(\bW_{n+1})=
\begin{cases}
\frac{11n^2+62n+56}{(4n+6)^2}; & \text{if $n$ is even},\\
\frac{44n^3+328n^2+1292n+678}{(4n+6)^3}; & \text{if $n$ is odd}.
\end{cases}$$
\end{theorem}

On reversing the colouring pattern of Mycielskian of complete graphs, mentioned in the previous section, we have the corresponding \textit{p.m.f} is given by
$$f(i)=
\begin{cases}
\frac{1}{2n+1}; & \text{if}\ i=1,\\
\frac{2}{2n+1}; & \text{if}\ i=2,3\ldots, n-1,n,\\
\frac{n}{2n+1}; & \text{if}\ i=n+1.
\end{cases}$$
\ni The $\chi^+$-chromatic parameters of $\bK_n$ are determined in the following theorem. 

\begin{theorem}
The $\chi^+$-chromatic mean of the Mycielski graph $\bK_n$ of the complete graph $K_n$ is given by
$$\dE(\bK_n)=\frac{3n^2+5n}{4n+2}$$
and the $\chi^+$-chromatic variance of $\bK_n$ is given by
$$\dV(\bK_n)=\frac{5n^4+10n^3-5n^2+14n}{3(4n+2)^3}.$$     
\end{theorem}

\section{Concluding Remarks}

Invoking Theorem \ref{Thm-2.1}, we note that as the order of the path increases, the $\chi$-chromatic mean of its Mycielski graphs converges to the number $\frac{7}{4}$. That is, the lower bound for the $\chi$-chromatic means of the family of Mycielskian of paths is $\frac{7}{4}$. The first connected Mycielskian of a path is $\bP_2$ and by Theorem \ref{Thm-2.1} its $\chi$-chromatic mean is $1.8$ and the $\chi$-chromatic mean of $\bP_3$ is $1.71$ and so on. Therefore, $1.75< \cE(\bP_n)\le 1.80$. 

Similarly, the $\chi$-chromatic variance converges to $\frac{11}{16}$ as the order of the paths increases. But for $\bP_2$, we have $\cV=0.56$ and for $\bP_3$, we have $\cV=0.48$, for $\bP_4$, we have $\cV=0.617$, for $\bP_5$, we have $\cV=0.56$ and so on. Hence, we can see that 
$\frac{11}{16}$ is the upper bound for the $\chi$-chromatic variances of the family Mycielskians of paths and hence we have $0.48\le \cV(\bP_n)<0.6875$. 

In a similar argument, we can show that that the range for $\chi^+$ chromatic mean of the Mycielski graphs of paths is $2.2\le \dE(\bP_n)<2.25$ and the range for their  $\chi^+$ chromatic variance is $0.487<\dV(\bP_n)<0.6875$.

Invoking Theorem \ref{Thm-2.3}, we note that as the order of the cycle increases, the $\chi$-chromatic mean of its Mycielski graphs also converges to the number $\frac{7}{4}$. That is, the lower bound for the $\chi$-chromatic means of the family of Mycielskian of paths is $\frac{7}{4}$. The first connected Mycielskian of a cycle is $\bC_3$ and by Theorem \ref{Thm-2.3} its $\chi$-chromatic mean is $2.0$ and the $\chi$-chromatic mean of $\bC_4$ is $1.77$ and so on. Therefore, $1.75< \cE(\bC_n)\le 2.0$. 

Similarly, the $\chi$-chromatic variance of $C_n$ also converges to $\frac{11}{16}$ as the order of the cycle increases. But note that $\cV(\bC_3)=1.11$, $\cV(\bC_4)=0.61$, $\cV(\bC_5)=0.97$ and so on. Hence, it can be seen that the $\chi$-chromatic variance of the Mycielskian of even cycles and increases gradually from $0.611$ to $0.6875$ and $\chi$-chromatic variance of the Mycielskian of odd cycles and decreases gradually from $1.14$ to $0.6875$. 

But, in view of Theorem \ref{Thm-3.2}, it can be noted that the bounds for the $\chi^+$-chromatic mean are different for even cycles and odd cycles. We can verify that the range for $\chi^+$ chromatic mean of the Mycielski graphs of even cycles is $0.22\le \dE(\bC_n)<2.25$, while that for the Mycielski graphs of odd cycles is $0.3\le \dE(\bC_n) <3.25$. 

It can also be seen that the $\chi^+$-chromatic variance of the Mycielskian of even cycles and increases gradually from $0.611$ to $0.6875$ and $\chi^+$-chromatic variance of the Mycielskian of odd cycles and decreases gradually from $1.11$ to $0.6875$.

Similarly, we can determine the bounds, if exist, for these two types of chromatic mean and variance for the Mycielski graphs of the other graphs classes also. 

\section{Scope for Further Studies}

In this paper, we have discussed the two types of chromatic parameters of the Mycielski graphs of certain basic graph classes. We also discussed the range of the numerical values of these parameters for certain graph classes. There are more open problems in this area which seem to be promising and interesting for further investigation. Determining the different colouring parameters of several other standard graph classes and their Mycielski graphs are yet to be settled. 

\ni We mention some of these open problems, we have identified during our present study.

\begin{problem}\textup{
Determine different colouring parameters of different integral powers of the Mycielskian of different graph classes.}
\end{problem}

\begin{problem}\textup{
Determine different colouring parameters of different operations products of the Mycielskian of various graph classes.}
\end{problem}

\begin{problem}\textup{
Determine chromatic parameters of the Mycielskian of various graph classes with respect to different graph colourings, like injective colouring, equitable colouring, $b$-colouring etc.}
\end{problem}

\begin{problem}\textup{
Determine certain other chromatic parameters such as skewness, kurtosis etc. with respect to different graph colourings, of various graph classes and discuss the symmetry of the probability distribution.}
\end{problem}

\begin{problem}\textup{
Verify the existence of the bounds for the colouring means and variances and determining these bounds, if exist, for different types of colourings of various graph classes and their associated graphs.}
\end{problem}

\begin{problem}\textup{
Note that the mean and variance with respect to different graph colouring corresponds to some sequences of positive numbers (possibly infinite). Analyse these sequences in respect of convergence (or divergence).}
\end{problem}

The concepts of colouring parameters can extensively be used for modelling different practical situations and for finding the most appropriate solutions to these problems. Hence, the studies in this area have wide scope for further investigation.

\section*{Acknowledgements}

The second author would like to dedicate this article to the memory of his mentor and motivator (Late) Prof. (Dr.) D. Balakrishnan, Founder Academic Director, Vidya Academy of Science and Technology, Thrissur, India. 

Authors would like to acknowledge gratefully the critical comments and suggestions of the referee which significantly improved the content and style of presentation of the paper.

\end{document}